\documentclass[12pt]{article}

\usepackage{multirow}

\usepackage{enumitem}
\usepackage{subfig}

\usepackage{lipsum} 
\usepackage{wrapfig}
\usepackage{graphicx}
\usepackage{ifthen}
\newboolean{PrintVersion}
\setboolean{PrintVersion}{false} 
\usepackage{amsmath,amssymb}
\usepackage{amsthm}
\usepackage{mathtools}
\usepackage{graphicx}
\usepackage{caption}
\usepackage{bbm}
\usepackage{color}
\usepackage{dcolumn}
\usepackage{bm}
\usepackage{float}
\usepackage[pagebackref=true]{hyperref} 
\renewcommand*\backref[1]{\ifx#1\relax \else (Cited on #1) \fi}\usepackage{apacite}
\definecolor{background-color}{gray}{0.98}
\definecolor{steelblue}{rgb}{0.27, 0.51, 0.71}
\definecolor{brickred}{rgb}{0.8, 0.25, 0.33}
\definecolor{bluegray}{rgb}{0.4, 0.6, 0.8}
\definecolor{amethyst}{rgb}{0.6, 0.4, 0.8}

\hypersetup{
    plainpages=false,       
    unicode=false,          
    pdftoolbar=true,        
    pdfmenubar=true,        
    pdffitwindow=false,     
    pdfstartview={FitH},    
    pdfnewwindow=true,      
    colorlinks=true,        
    linkcolor=steelblue,         
    citecolor=bluegray,        
    filecolor=magenta,      
    urlcolor=cyan           
}
\ifthenelse{\boolean{PrintVersion}}{   
\hypersetup{	
    citecolor=black,%
    filecolor=black,%
    linkcolor=black,%
    urlcolor=black}
}{} 

\theoremstyle{plain}
\newtheorem{theorem}{Theorem}[section]
\newtheorem{lemma}[theorem]{Lemma}
\newtheorem{corollary}[theorem]{Corollary}

\newtheorem{proposition}[theorem]{Proposition}
\newtheorem{remark}[theorem]{Remark}
\theoremstyle{definition}

\newenvironment{classification}
  {\noindent \textbf{AMS classification: }}
  
\newenvironment{keywords}
  {\noindent \textbf{Keywords: }}

\newcommand*{\intersect}{\cap}
\newcommand*{\union}{\cup}
\newcommand*{\bigintersect}{\bigcap}
\newcommand*{\bigunion}{\bigcup}
\newcommand{\cardinality}[1]{|{#1}|}

\newcommand{\bigcardinality}[1]{\left|{#1}\right|}

\newcommand*{\powerset}[1]{\mathcal{P}({#1})} 

\newcommand{\family}[1]{{\cal #1}}

\let\emptyset\varnothing

\newcommand{\Nset}[1]{[#1]} 
\newcommand{\mSubsetsN}[2]{\binom{\Nset{#1}}{#2}}
\newcommand{\set}[1]{\mathcal{#1}}
\newcommand{\MinIntersecting}{\set{M}_{min}}
\newcommand{\MaxIntersecting}{\set{M}_{max}}

\newcommand{\intersectingFamily}[1]{\mathcal{I}}
\newcommand{\pathIntersectingFamily}[1]{\mathcal{I}_p}

\title{The Clique Structure of 
       Johnson Graphs }
\author{Pavel Shuldiner \& R. Wayne Oldford\\University of Waterloo}
\begin{document}
\maketitle
\begin{abstract}
Motivated by an approach to visualization of high dimensional statistical data given in  \citeA{navGraphs2011}, this work examines the clique structure of $J_n(m, m-1)$ Johnson graphs.  Cliques and maximal cliques are characterized and proved to be of one of only two types.  These types are characterized by features of the intersection and of the union of the subsets of $\Nset{n} = \{1,2, \ldots, n\}$ which define the vertices of the graph.  Clique numbers and clique partition numbers follow. 
The results on Johnson graphs are connected to results on intersecting families of sets related to extremal set theory. 
\end{abstract}

\section{Introduction}

Let $\Nset{n}$ denote the set of the first $n$ positive integers $\{1,2,\ldots, n\}$ and $\mSubsetsN{n}{m}$  the collection of all subsets of $\Nset{n}$ of size $m$. 
The Johnson graph $G = (V, E)$, denoted $J_n(m,m-1)$, has $\binom{n}{m}$ vertices/nodes $v \in V$, each labelled by a unique set $\nu(v) \in \mSubsetsN{n}{m}$, and  $\binom{n}{m}\frac{m(n-m)}{2}$ undirected  edges $e_{ij} \in E$ between distinct nodes $v_i$ and $v_j$ if, and only if, $\cardinality{\nu(v_i) \intersect \nu(v_j)} = m-1$.    Figure \ref{fig:johnson} shows two examples -- $J_4(2,1)$ in (a) and $J_5(3,2)$ in (b).

Our motivation for studying Johnson graphs originates with the problem of visual exploration of high dimensional data.  If there are $N$ observations on $n$ variables in a statistical data set, then the data can be thought of as a set of $N$ points in $n$ dimensional real space.  Such data are most naturally viewed in $2$ and $3$ dimensional scatterplots.  \citeA{navGraphs2011} suggest interpreting a $J_n(2,1)$ Johnson graph as having $2$-dimensional spaces as its nodes (defined by the indices of the variables) and, as its edges, $3$-dimensional spaces defined by the union of the variables defining the adjacent nodes (e.g., see Figure \ref{fig:johnson}(a) for $n=4$).  The Johnson graph provides a ``navigation graph'' for exploring high dimensional point clouds along lower dimensional trajectories.  Following a path along the graph traverses from one $2d$-space to another along $3d$-transitions.  Dynamic $3d$-scatterplot rotations from one $2d$-space to another have been used effectively in data analysis \cite<e.g., see>{oldford2011visual, loon}, as has static displays of large numbers of $2d$-scatterplots laid out by following paths in the $J_n(2,1)$ with neighbouring scatterplots sharing a common axis \cite<e.g., see>{hofertoldford2018, zenplotsJSS}.  \citeA{navGraphs2011} also consider using paths along $J_n(m, m-1)$ Johnson graphs in conjunction with more complex visualizations to perceive structure in point clouds of dimension $m \geq 3$.  Understanding the clique structure of the $J_n(m, m-1)$ for arbitrary $m$ will help data analysts better understand, and make use of, lower dimensional regions of the full $n$-dimensional space of the data. 

In what follows, focus is on characterizing the structure of cliques, particularly maximal cliques in the $J_n(m, m-1)$ Johnson graph.  
Section \ref{sec:prelim} begins with some preliminary results for the $J_n(2,1)$ Johnson graph of interest in our motivating example.  This section illustrates the logic, and provides the base case, for many of the more general results developed in Section \ref{sec:general-results} for the $J_n(m, m-1)$ Johnson graph. 
Results in both sections are derived without reference to the motivating example.
These include the characterization of maximal cliques (there are only two types) of a $J_n(m, m-1)$, from which follows the clique number.  Section \ref{sec:extending} characterizes the nature of any $r$-clique, from which the clique partition number follows in Section \ref{sec:partition_num}.

Johnson  ``graphs are important because they enable us to translate many combinatorial problems about sets into graph theory'' \cite[p. 9]{godsil2001algebraic}.   
Section \ref{sec:discussion} discusses the results of earlier sections in the context of the \text{intersecting families} of sets from extremal set theory \cite{ExtremalSetBook}.
\citeA{shuldiner2022many} showed how intersecting families of sets can also be related to cliques in a clique cover. 
The last section ends with some discussion on the implications of the results in the context of the motivating example of statistical data analysis.

\section{Preliminaries}
\label{sec:prelim}
%
Figure \ref{fig:johnson}
\begin{figure}[h]
\begin{center}
\begin{tabular}{cp{0.1\textwidth}c}
\includegraphics[width = 0.30\textwidth]{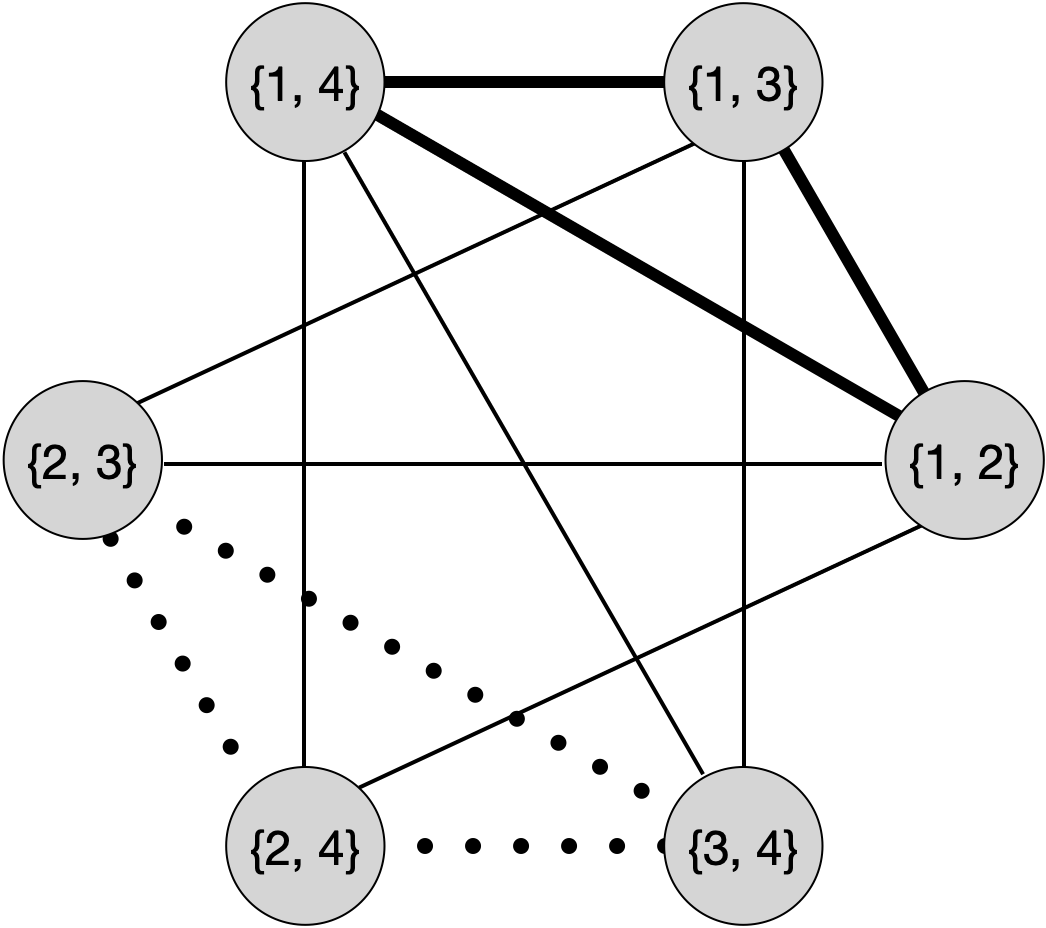} & &
\includegraphics[width = 0.375\textwidth]{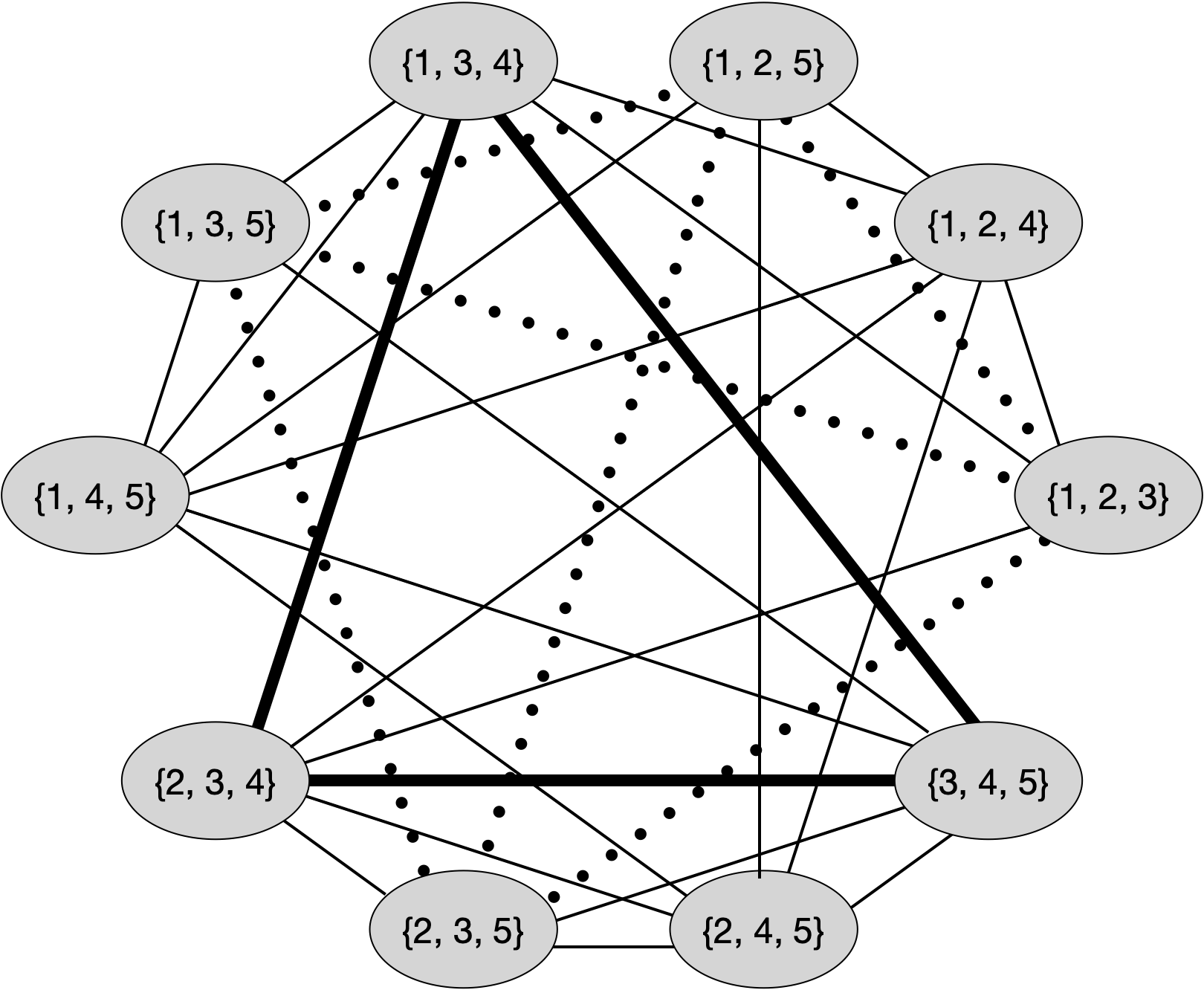} \\
(a) $J_4(2,1)$   & &
(b) $J_5(3,2)$\\
\end{tabular}
\end{center}
  \caption{Two separate Johnson $J_n(m, m-1)$ graphs with label sets $\nu(v)$ shown on each node $v$.  Nodes are identified as $v_1$, $v_2$, \ldots, beginning from the right most node in each graph and from there in counter-clockwise order. Two maximal cliques are marked on each.}
\label{fig:johnson}
\end{figure}
shows two examples of  $J_n(m, m-1)$ graphs for (a) $n=4$, $m=2$, and (b) $n=5$, $m=3$.  Identifying the nodes of $J_4,(2,1)$ as $v_1, \ldots, v_6$ in counter-clockwise order beginning from the rightmost node of Figure \ref{fig:johnson}(a) gives $\nu(v_1) = \{1, 2\}$, $\nu(v_2) = \{1, 3\}$, \ldots, $\nu(v_6) = \{3, 4\}$.  Similarly,  beginning from the rightmost node of Figure \ref{fig:johnson}(b), and moving counter-clockwise, yields label sets $\nu(v_1) = \{1, 2, 3\}$, $\nu(v_2) = \{1, 2, 4\}$, \ldots, $\nu(v_{10}) = \{3, 4, 5\}$ for $J_5(3,2)$.
 
For any subgraph $H \subseteq G$, the \textit{intersection of $H$} will refer to the set
\[S = \intersect_{v \in V(H)} ~\nu(v) ~\text{or, simply, }~ S = \intersect_{v \in H} ~\nu(v), \]
the intersection of the label sets for the nodes $V(H)$ of $H$.  For example, in the $J_4(2,1)$ Johnson graph of Figure \ref{fig:johnson}(a), consider the subgraphs $H_1$, $H_2$, and $H_3$ induced by vertex sets $V(H_1)=\{v_1, v_2, v_3\}$ (shown with thick edges in Fig. \ref{fig:johnson}(a)), $V(H_2) =\{v_4, v_5, v_6\}$ (shown with thick dotted edges in Fig. \ref{fig:johnson}(a)), and $V(H_3) =\{v_2, v_3, v_5, v_6\}$ (not shown), respectively.  The intersection
\begin{itemize}
\item 
of $H_1$ is $S_1 = \intersect_{v \in H_1} ~\nu(v) = \{1\}$,
\item 
of
 $H_2$ is $S_2 = \intersect_{v \in H_2} ~\nu(v) = \emptyset$, and
\item 
of $H_3$ is $S_3 = \intersect_{v \in H_3} ~\nu(v) = \emptyset$.
\end{itemize}

Of special interest is the relationship between this intersection set and cliques, $H$ of $G$ (e.g., $H_1$ and $H_2$ above; not $H_3$).  For example, for a clique $H$ to be  \textit{maximal} (i.e., no larger clique contains $H$) in a $J_n(2,1)$ graph, it is easy to see that the size of its intersection set is either 1 (e.g., $\cardinality{S_1}=1$) or 0 (e.g., $\cardinality{S_2}=0$), as shown below in Lemma \ref{lemma:johnson_base_intersections}.
\begin{lemma}
\label{lemma:johnson_base_intersections}
For $G = J_n(2,1)$, the size of the intersection of node label sets on any maximal clique in $G$ is at most 1. 
\end{lemma}
\begin{proof}
Let $H$ be a maximal clique in $G = J_n(2,1)$ and $S = \intersect_{v\in H}\nu(v).$
Since every node label set has 2 elements, and any two nodes intersect in exactly one element, it follows that $|S|\leq 1$. 
\end{proof}
Both 0 and 1 are possible sizes for the intersection set $S$ of a maximal clique in $G = J_n(2,1)$.   Moreover, maximal cliques in $J_n(2,1)$ are only of two possible sizes according to the size of their intersection set $H$.  This is shown in Lemma \ref{lemma:johnson_base_clique_size} below.
\begin{lemma}
\label{lemma:johnson_base_clique_size}
For any maximal clique $H$ of $G = J_n(2,1)$, with intersection set $S$, for $n \ge 3$
\[
\cardinality{H} = \left\{ \begin{array}{ccl}
3 & \iff & \cardinality{S} = 0\ \\
& & \\
(n-1) & \iff & \cardinality{S} = 1
\end{array}
\right.
\]
\end{lemma}
\begin{proof}
Begin with the smallest non-trivial clique $H \subset G = J_n(2,1)$ of size two with vertex label sets $\nu(v_1) = \{a, b\}$ and $\nu(v_2) = \{b, c\}$, for distinct numbers $a, b, c  \in \Nset{n}$.  $H$ is not maximal for it can be extended by a single node $v_3$ in one of only two possible ways, as shown below
\begin{center}
\includegraphics[width = 0.5\textwidth]{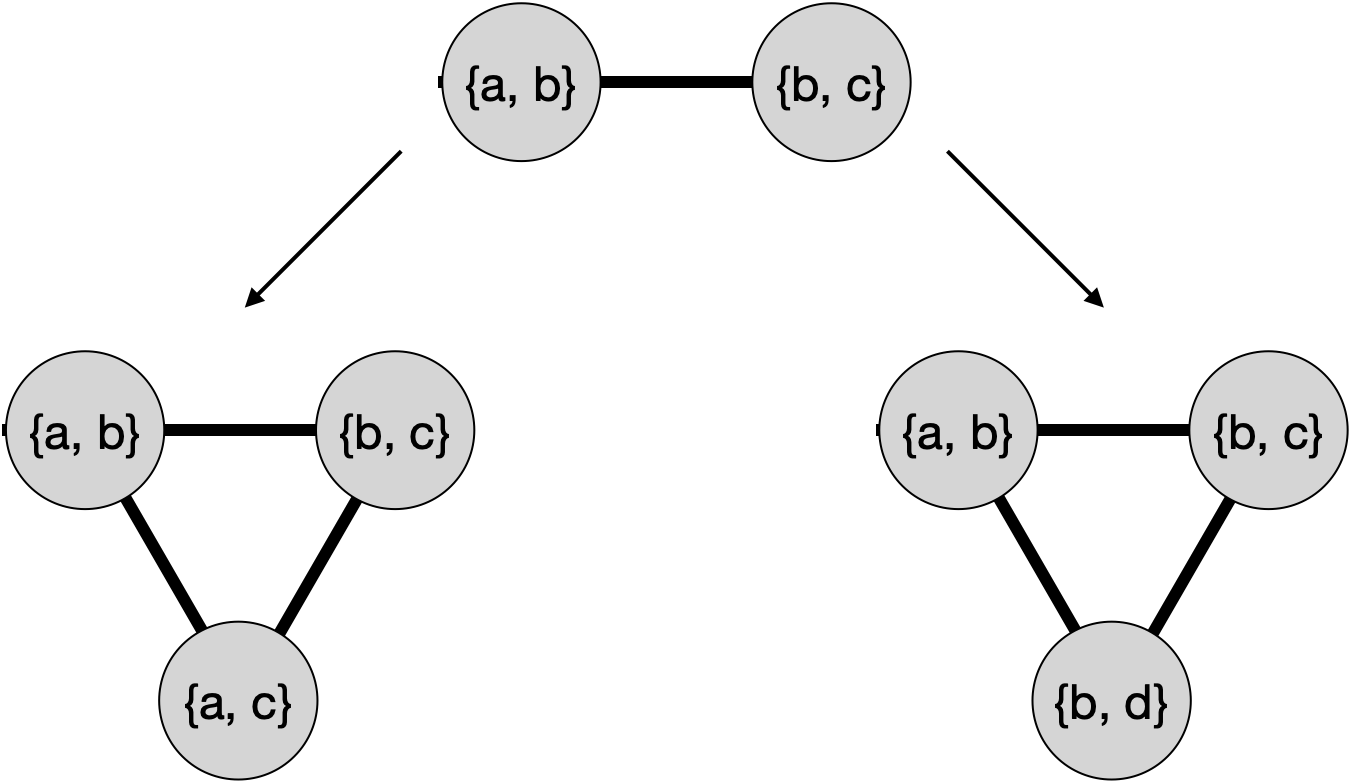}
\end{center}
for a fourth distinct number $d \in \Nset{n}$.  

Consider first the leftmost triangle.  
Its intersection set $S = \{a, b\} \intersect \{b, c\} \intersect \{a, c \} = \emptyset$ and $\cardinality{S} = 0$.  No fourth node, $v_4$, can be added and the clique maintained, so this triangle is also a maximal clique.  To see this, recall that any node $v$ adjacent to both $v_1$ and $v_2$ must have label set of either $\{a,c\}$ or $\{b, d\}$, as shown above.  But the first is already in the triangle and the second has no intersection with $\{a, c\} = \nu(v_3)$, meaning $v_4$ cannot be adjacent to $v_3$.  So, this triangle, with $\cardinality{S} = 0$, cannot be extended into a larger clique and, hence, must be maximal of size 3.

Now consider the rightmost triangle.  This clique has intersection $S = \{b\}$ giving $\cardinality{S} = 1$, but the clique is maximal only when $n=3$.
For $n > 4$, any node $v$ with label $\nu(v) = \{b, e\}$, where $e \in \Nset{n}$ is distinct from $a$, $b$, $c$, and $d$, will be adjacent to \textit{all} nodes in the triangle.  
There are exactly $(n-4)$ such choices remaining in $\Nset{n}$ to be paired with $b$ in a label set.  
The clique can therefore grow maximally to size $(n-4) + 3 = (n-1)$ with intersection set $S=  \{b\}$ of size $\cardinality{S} = 1$.
\end{proof}

Together, Lemmas \ref{lemma:johnson_base_intersections} and \ref{lemma:johnson_base_clique_size} yield the clique number, the size of the maximum clique in $G = J_n(2,1)$,
denoted $\omega(J_n(2,1))$, as follows:
\begin{corollary}
\label{cor:clique_num_J21}
The clique number of $G = J_n(2,1)$, for $n \ge 3$,  is 
\[\omega(J_n(2,1))= max \{n-1, 3\}.\]
\end{corollary}

According to Lemma \ref{lemma:johnson_base_clique_size}, there are two different types of maximal cliques in $J_n(2,1)$.  One set, say $\MinIntersecting$, contains maximal cliques $H \subset J_n(2,1)$ having \textit{minimal intersection} set of size $\cardinality{S} = 0$ (with each $\cardinality{H} = 3$ for $H \in \MinIntersecting$); the other, say $\MaxIntersecting$, contains those maximal cliques having \textit{maximal intersection set} of size $\cardinality{S} = 1$ (with each $\cardinality{H} = n-1$ when $n \ge 4$ for every $H \in \MaxIntersecting$).  The number of maximal cliques of each type is $\cardinality{\MinIntersecting} = \binom{n}{3}$ and $\cardinality{\MaxIntersecting} = \binom{n}{1}$.

Figure \ref{fig:johnson}(b) suggests that similar results could exist for the Johnson graph $J_n(m, m-1)$ more generally.  There, two different types of maximal cliques are shown for $G = J_5(3,2)$.  One, shown with dotted line edges, is a $4$-clique with intersection set
$S = \{1, 2, 3\} \intersect \{1, 2, 5\} \intersect \{1, 3, 5\} \intersect \{2, 3, 5\} = \emptyset$ of size $\cardinality{S} = 0$ is in keeping with Lemma \ref{lemma:johnson_base_clique_size} identifying a maximal clique seemingly in $\MinIntersecting$ for a $J_5(3,2)$.
Another, shown by thick solid line edges, is of size three and has intersection set $S = \{1, 3, 4\} \intersect \{2, 3, 4\} \intersect \{3, 4, 5\} = \{3, 4\}$ of size $\cardinality{S} = 2$ which is like  $\MaxIntersecting$ in that its intersection set also appears to be of maximum size, though this time 2 instead of 1.
More generally, it turns out that maximal cliques in any Johnson graph $G = J_n(m, m-1)$ either have an intersection set of size $\cardinality{S} = 0$ or $\cardinality{S} = m - 1$. 
This is proved as   
Theorem \ref{thm:johnson_max_clique_intersection} in the next section. 


\section{General results}
\label{sec:general-results}
As with Lemma \ref{lemma:johnson_base_intersections}, the size of an intersection set for any maximal clique of a Johnson graph $G = J_n(m, m-1)$ cannot be larger than $m-1$, given that is the size of the intersection of node label sets for a single edge.  The main result of this section, analogous to Lemma \ref{lemma:johnson_base_clique_size}, proves that this \textit{maximum} size, $m-1$, and the \textit{minimum} size, 0, are the \textit{only} values possible for the size of the intersection set of a maximal clique in $J_n(m, m-1)$.

In Lemma \ref{lemma:johnson_base_clique_size}, the types of maximal cliques for the simplest case of $J_n(2,1)$ were found by beginning with a clique of size two and seeing how it might be expanded by adding nodes.  There were only two possible ways to do this, each leading to a different type of maximal clique.  Here, we follow the same reasoning, but for vertices and edges from a $J_n(m, m-1)$ graph.  The figure and proof of Lemma \ref{lemma:johnson_base_clique_size} guide the intuition in this more general case.

We begin with the case corresponding to the left most diagram of Lemma \ref{lemma:johnson_base_clique_size}.  There, a third node, $v_3$, was added by  selecting the elements for its label set, $\nu(v_3) = \{a, c\}$ from the \textit{union} of the label sets of the first two vertices $v_1$ and $v_2$, that is $\nu(v_3) \subset B =  \nu(v_1) \union \nu(v_2)$.  This choice had repercussions in Lemma \ref{lemma:johnson_base_clique_size} in that the intersection set was null for $J_n(2,1)$ and the clique could not be enlarged past size 3 as in the dotted clique of Figure \ref{fig:johnson} (a).  
For the dotted clique of Figure \ref{fig:johnson} (b), however, the set $B$ for two vertices from a $J_n(3,2)$ is larger so that the clique can be enlarged to include a fourth node. 
In both dotted clique examples of Figure \ref{fig:johnson} the intersection of the label sets is null.

The next proposition characterizes the intersection sets for maximal cliques, $H$, formed in this way from a Johnson graph $G = J_n(m. m-1)$.

\begin{proposition}
\label{prop:minimally_intersecting_maximal_cliques}
Let $G = J_n(m,m-1)$ be a Johnson graph with $n\geq m+1$ and let $H$ be a maximal clique in $G$. If $B$ denotes the union $$B := \bigcup_{v\in H}\nu(v),$$ 	
then $\bigintersect_{v\in V(H)}\nu(v) = \emptyset$ if, and only if, $\nu(v_1) \cup \nu(v_2) = B$ for any $v_1, v_2$ distinct in $V(H)$. 
\end{proposition}
\begin{proof}
	Suppose that for any distinct $v_1, v_2 \in V(H)$, $\nu(v_1) \cup \nu(v_2) = B$. It follows that for any node $v \in V(H),$ $\nu(v)$ is an $m-$subset of the $(m+1)-$set $B$. Moreover, since every $m-$subset of $B$ corresponds to a node adjacent to every node in $H$, it follows that 
	\[ \nu(V(H)) = \{A: A\subset B, |A| = m\}.\]
For any $i \in B$, the node with label $\{j \in B: j \neq i\}$ eliminates $i$ from the intersection $\bigintersect_{v\in V(H)}\nu(v)$. Thus, the intersection $\bigintersect_{v\in V(H)}\nu(v)$ must be empty.

Conversely, suppose that $\bigintersect_{v\in V(H)}\nu(v) = \emptyset$. If $\nu(v_1) \cup \nu(v_2) \neq B$ for some $v_1, v_2\in V(H)$, then there exists some $v_3\in V(H)$ such that $x_3 \in \nu(v_3)$ and $x_3 \not\in \nu(v_1)\cup \nu(v_2)$. It follows that $v_3$ satisfies 
\[\nu(v_3) = \left(\nu(v_1)\intersect \nu(v_2)\right) \cup \{x_3\}.\]
Thus, $v_1, v_2$ and $v_3$ satisfy the hypothesis of Proposition \ref{prop:intersection_of_clique} and it follows that $\intersect_{v\in V(H)}\nu(v)$ is a set of size $m-1$, a contradiction.

\end{proof} 

It would appear, then, that, if we build up a maximal clique in this way, we end with one whose intersection set $S = \emptyset$ is of size zero.  This being the smallest possible intersection set, $\MinIntersecting$ could again denote the set of such maximal cliques in $J_n(m, m-1)$.

Returning to the intuition followed in the figure of Lemma \ref{lemma:johnson_base_clique_size}, consider how vertices were added when taking the righthand choice.  The label set of any third vertex would be formed from the intersection $\nu(v_1) \intersect \nu(v_2)$, necessarily of size $m-1$ in $J_n(m, m-1)$, joined by any element of $\Nset{n}$ not already appearing in the union $\nu(v_1) \union \nu(v_2)$.  Following the same logic, vertices could be added providing the intersection set remained of size $m-1$ until all remaining elements of $\Nset{n}$ were exhausted, that is, until $\nu(v_1)\union \nu(v_2) \union \cdots \union \nu(v_r) = \Nset{n}$.  
For a Johnson $J_n(m, m-1)$, the number of vertices for such a clique would be $r = n - (m-1)$ (e.g., the thick solid line maximal cliques shown in Figures \ref{fig:johnson}).
Cliques formed in this fashion, would have largest possible intersection set of size $m-1$ and so could, again, be denoted $\MaxIntersecting$.

The next proposition shows that building a maximal clique $H$ of a Johnson graph $G = J_n(m. m-1)$ in this way can only lead to one having largest intersection set of size $m-1$ and, hence, to $H \in \MaxIntersecting$.
\begin{proposition}
\label{prop:intersection_of_clique}	
Let $H$ be a maximal clique in $G = J_n(m, m-1)$. If $|\nu(v_1) \intersect \nu(v_2) \intersect \nu(v_3)| = m-1$ for some distinct nodes $v_1, v_2, v_3 \in V(H)$, then 
	\[\bigintersect_{v\in V(H)}\nu(v) = \nu(v_1)\intersect \nu(v_2).\] 
\end{proposition}

\begin{proof}
Suppose that $I := \nu(v_1) \intersect \nu(v_2) \intersect \nu(v_3)$ for some distinct nodes $v_1, v_2$ and $v_3$ in $H$, and write $\nu(v_j)$ in the form $\{x_j\} \cup I$ for $j=1,2,3$. If $\bigintersect_{v\in V(H)}\nu(v) \neq I$, then there must be some $i\in I$ for which $i \not \in \nu(u_i)$ for some $u_i \in V(H)$.

Then since $u_i \in V(H)$ and $H$ is a clique, $u_i \sim v_1, v_2, v_3$. Therefore, $\nu(u_i)$ must contain $x_j$ and an $(m-2)-$subset $I_j$ from $I$, for $j = 1, 2, 3$. In other words,
\[\nu(u_i) = I_1\cup I_2 \cup I_3 \cup \{x_1, x_2, x_3\}.\]
Since $I$ is disjoint from $\{x_1, x_2, x_3\}$, so are the $I_1, I_2, I_3$. Since the nodes $v_1, v_2, v_3$ are distinct, the variables $x_1, x_2, x_3$ are distinct. Consequently, 
\[|\nu(u_i)| = |I_1 \cup I_2 \cup I_3| +3 \geq m-2 + 3= m+1,\]
and such a node cannot exist in $J_n(m, m-1)$.
\end{proof}

Should a maximal clique $H \in J_{n}(m, m-1)$ belong to one of the two classes $\MinIntersecting$ and $\MaxIntersecting$, Propositions \ref{prop:minimally_intersecting_maximal_cliques} and \ref{prop:intersection_of_clique} provide information about $H$ which is summarized in the following remark.
\begin{remark}
\label{remark:classes}
Maximal cliques $H \in J_{n}(m, m-1)$ have the following characteristics unique to which class, $\MinIntersecting$ or $\MaxIntersecting$, they belong.
\begin{itemize}
	\item $H \in \MinIntersecting$:
	\begin{itemize}
	\item The intersection set $S = \intersect_{v \in V(H)} \nu(v) = \emptyset$ with $\cardinality{S} = 0$.
	\item
	Node labels of all vertices $v_j \in V(H)$ are distinct and of the form $\nu(v_j) = B \setminus \{x_j\}$ for some $B \subset \Nset{n}$ of size $\cardinality{B} = m + 1$ and all $x_j \in B$.
	\item The size of the maximal clique $H$ is $\cardinality{H} = \cardinality{B} = m + 1$.
	\item For any distinct $v_i$, $v_j$ in $V(H)$, $\nu(v_i) \union \nu(v_j) = B$.
	\item The number of distinct maximal cliques in $\MinIntersecting$ is $\binom{n}{\cardinality{B}} = \binom{n}{m+1}$.
	\end{itemize}
	\item $H \in \MaxIntersecting$:
	\begin{itemize}
	\item The intersection set $S = \intersect_{v \in V(H)} \nu(v)$ has size $\cardinality{S} = m-1$.
	\item
	Node labels of all vertices $v_j \in V(H)$ are distinct and of the form $\nu(v_j) = A \union \{x_j\}$ for some $A \subset \Nset{n}$ of size $\cardinality{A} = m - 1$ and all $x_j  \in \Nset{n}\setminus A$.
	\item The size of the maximal clique $H$ is $\cardinality{H} =  \cardinality{\Nset{n}} - \cardinality{A} = n - m + 1$.
	\item For any distinct $v_i$, $v_j$ in $V(H)$, $\nu(v_i) \intersect \nu(v_j) = A = S$.
	\item The number of distinct maximal cliques in $\MaxIntersecting$ is $\binom{n}{\cardinality{A}} = \binom{n}{m-1}$.
	\end{itemize}
\end{itemize}
\end{remark}

Following the logic of Lemma \ref{lemma:johnson_base_clique_size}, Propositions \ref{prop:minimally_intersecting_maximal_cliques} and \ref{prop:intersection_of_clique} demonstrate at least two distinct classes of maximal cliques exist in $J_n(m, m-1)$, $\MinIntersecting$ and $\MaxIntersecting$.
That these are the \textit{only} types of maximal cliques  in a Johnson graph $G = J_n(m, m-1)$ is proved in Theorem \ref{thm:johnson_max_clique_intersection} (by induction on $m$, with Lemma \ref{lemma:johnson_base_clique_size} providing the initial case). 

\begin{theorem}
\label{thm:johnson_max_clique_intersection}
	Let $H$ be a maximal clique in the Johnson graph $G = J_n(m, m-1)$ for $n > m \ge 2$ and let 
	$S = \intersect_{v\in H}\nu(v)$ be the intersection set of the node labels of $H$. Then, $\cardinality{S} \in \{0, (m-1)\}$.
\end{theorem}
\begin{proof}
The proof proceeds by induction on $m$.

Suppose, that there exists $m_0 \geq 2$ such that, for all $n_0 > m_0$, any maximal clique $H_0$ in $J_{n_0}(m_0, m_0-1)$ with intersection set $S_0 = \intersect_{v\in H_0}\nu(v)$ has $\cardinality{S_0} \in \{0, m_0 -1\}$. For a proof by induction, we need to show that this implies that for $m_1 = m_0 +1$, any maximal clique $H_1$ in $J_{n_1}(m_1, m_1 -1)$, for all $n_1 > m_1$, must have intersection set of size $\cardinality{S_1} \in \{0, m_1 -1 \}$.

The inductive step is proved by contradiction.  The size of the intersection set of a maximal clique $H_1$ in $J_{n_1}(m_1, m_1 -1)$ is at most $m_1 -1$, so we assume the intersection set size $\cardinality{S_1} = s$ with $0 < s < m_1 - 1 = m_0$.  Without loss of generality, we may take the intersection $S_1$ to be
\[S_1 = \left\{ (n_1 - s + 1),  (n_1 - s + 2), \ldots, n_1 \right\}. \]

Knowing that the node label sets of $H_1$ have intersection $S_1$ of size $s$ and that $H_1$ induces a maximal clique in $J_{n_1}(m_1, m_1 -1)$, the label set for any vertex $v_i \in H_1$ can be expressed as  $\nu(v_i) = I_i \union S_1$ where $I_i \subset \Nset{n_1}$ with $I_i \intersect S_1 = \emptyset$ and $\cardinality{I_i \intersect I_j} = m_1 -1 -s$ for $i \ne j$ and  $v_i, v_j \in H_1$.   

This knowledge allows us to construct vertices $v^-_j \in J_{n_1 -1}(m_1 - 1, m_1 -2) = J_{n_0}(m_0, m_0 -1)$ with $n_0  = n_1 - 1 >  m_1 -1 = m_0$ by simply removing $n_1 \in S_1$ from the node label sets of all vertices in $H_1$.  That is, for every $v_i \in V(H_1)$,
 define $v^-_i \in V(J_{n_0}(m_0, m_0 -1))$ to have node label set $\nu(v^-_i) = \nu(v_i)\setminus \{n_1 \}$. 

Now consider the graph, $H_0$,  induced in $J_{n_0}(m_0, m_0 -1)$ by the vertices $v^-_i$, so defined.
It is easy to see that $H_0$ is a clique in $J_{n_0}(m_0, m_0 -1)$ -- the label sets of its vertices are $\nu(v^-_i) = I_i \union S_1 \setminus \{n_1\}$ with $\nu(v^-_i) \intersect \nu(v^-_j) =  (I_i \intersect I_j) \union S_1 \setminus \{n_1\}$ yielding cardinalities of $(m_1 - s) + (s -1) = m_1 - 1 = m_0$ and $(m_1 - s - 1) + (s -1) = m_1 - 2 = m_0 -1$, respectively.

To see that $H_0$ is also maximal, suppose that it is not.  Then, there is some vertex $v \in J_{n_0}(m_0, m_0 -1)$ which is not in $H_0$ but is adjacent to every vertex $v^-_i \in H_0$. 
This adjacency implies 
\[ \cardinality{\nu(v) \intersect (I_i \union S_1 \setminus \{n_1\})} = m_0 -1\]
for all $v^-_i \in H_0$.

We construct a vertex $v^+ \in J_{n_1}(m_1, m_1 -1)$ having label set $\nu(v^+) = \nu(v) \union \{n_1\}$. 
Since $v$ is adjacent to every $v^-_i \in H_0$, its intersection with each is of size $m_0 -1 = \cardinality{\nu(v) \intersect (I_i \union S_1 \setminus \{n_1\})}$. It follows that for $v_i \in H_1$,
\[ 
\begin{array}{rcl}
\nu(v^+) ~\bigintersect~ \nu(v_i) &=&
\nu(v^+) ~\bigintersect ~(I_i \union S_1)\\
& &\\
& =& 
\nu(v^+) ~\bigintersect ~[(I_i \union (S_1 \setminus \{n_1\}))~\bigunion~  \{n_1\}]\\
& &\\
& =&[\nu(v^+)  ~\bigintersect~ (I_i \union (S_1 \setminus \{n_1\}))]~ \bigunion ~[\nu(v^+)  \intersect \{n_1\}]\\
& &\\
& =&[\nu(v^+)  ~\bigintersect~ (I_i \union (S_1 \setminus \{n_1\}))] ~\bigunion~ [\{n_1\}]\\
& &\\
& =&[(\nu(v) \union \{n_1\})  ~\bigintersect~ (I_i \union (S_1 \setminus \{n_1\}))] ~\bigunion~ \{n_1\}\\
& &\\
& =&\left[ (\nu(v)  ~\bigintersect~ (I_i \union (S_1 \setminus \{n_1\}))) ~\bigunion ~ (\{n_1\} ~\bigintersect~ (I_i \union (S_1 \setminus \{n_1\})))\right]~\bigunion~ \{n_1\}\\
& &\\
& =&\left[ (\nu(v)  ~\bigintersect~ (I_i \union (S_1 \setminus \{n_1\}))) ~\bigunion ~ \emptyset\right]~\bigunion~ \{n_1\}\\
& &\\
& =&\left[ \nu(v)  ~\bigintersect~ (I_i \union (S_1 \setminus \{n_1\})) \right]~\bigunion~ \{n_1\}.\\
\end{array}
 \]
Now, in square brackets, the left set of the union does not contain $n_1$ and is of known cardinality $m_0 -1$.  It follows, then, that
\[\bigcardinality{\nu(v^+) ~\bigintersect~ \nu(v_i)} = \bigcardinality{ \nu(v)  ~\bigintersect~ (I_i \union (S_1 \setminus \{n_1\})) }~+~ \bigcardinality{\left\{n_1\right\}} = (m_0-1) + 1 = m_1 -1.
\]
Hence, $v^+$ is adjacent to every vertex $v_i \in H_1$, and $H_1$ can be extended to a larger clique  in $J_{n_1}(m_1, m_1 -1)$ -- a contradiction since $H_1$ was assumed to be maximal.  It follows, then, that no such node, $v \in J_{n_0}(m_0, m_0 -1)$, exists which extends the clique $H_0$ and, hence, that $H_0$ must be maximal.

By construction, the intersection set of $H_0$ is $S_0 = S_1 \setminus \{n_1\}$ and is of size
$\cardinality{S_0} = s -1$.  And, because $H_0$ is a maximal clique in $J_{n_0}(m_0, m_0 -1)$, by the inductive hypothesis $\cardinality{S_0}$ is either 0 or $m_0-1$.  If the latter, then $s = m_0$ is outside the bounds assumed and we have a contradiction.  If the former, then, $s=1$ and, by Proposition \ref{prop:minimally_intersecting_maximal_cliques}, $H_0 \in \MinIntersecting$ so $S_0 = \emptyset$ and $\cardinality{S_0} = 0$  -- again, a contradiction.  

It follows that intersecting sets of a maximal clique $H_1 \in J_n(m_1, m_1 - 1)$ with $m_1 = m_0 + 1$ must have size 0 or $m_1 - 1$, if it is the case that intersecting sets of a maximal clique $H_0 \in J_n(m_0, m_0 - 1)$  must be of size 0 or $m_0 - 1$.  The proof by induction is complete by noting that, by Lemma \ref{lemma:johnson_base_clique_size}, the inductive hypothesis holds for $m_0 = 2$.

\end{proof}

Theorem \ref{thm:johnson_max_clique_intersection} proved that every maximal clique $H \in J_n(m,m-1)$ has either the minimal or maximal intersection set possible.  That is, either $H \in \MinIntersecting$, or $H \in \MaxIntersecting$.  If the former, then $\cardinality{H} = m + 1$; if the latter, then  $\cardinality{H} = n - m + 1$ (see Remark \ref{remark:classes}).  As a consequence,  we obtain the clique number of $J_n(m,m-1)$ for all $n\geq m+1$.
\begin{corollary}
\label{cor:clique_num_johnson}
	The clique number $\omega(J_n(m,m-1))$ of the Johnson graph $J_n(m,m-1)$ is given by
	\[\max(m+1, n-m+1),\]
	whenever $n \geq m+1$.  
\end{corollary} 

\noindent
Rewriting, it follows from Corollary \ref{cor:clique_num_johnson}, that the clique number is
\[
\omega(J_n(m, m-1)) = \left\{
   \begin{array}{lcl}
   m +1 & \text{if} & m + 1 \leq n \le 2m \\
   && \\
   n - m + 1 & \text{if} & 2m \le n
   \end{array}
\right.
\]
and is undefined otherwise.


\section{Extending an $r$-clique}
\label{sec:extending}
Given some clique $C_r \subset J_n(m, m-1)$ of size $\cardinality{C_r} = r$, what can be said about the maximal cliques $H \subset J_n(m, m-1)$ that contain it?

We begin with edges ($r = 2$). As the figure in the proof of Lemma \ref{lemma:johnson_base_clique_size} suggests, every edge in $J_n(m, m-1)$ can appear in one clique from $\MinIntersecting$ and one from $\MaxIntersecting$.  Proposition \ref{prop:edge_extensions} shows that each edge can appear in \textit{only one} maximal clique in each of $\MinIntersecting$ and $\MaxIntersecting$.

\begin{proposition}
\label{prop:edge_extensions}
	Each edge of $J_{n}(m,m-1)$ will belong to precisely one maximal clique $H_{min} \in \MinIntersecting$ and to precisely one maximal clique $H_{max}  \in \MaxIntersecting$.
\end{proposition}
\begin{proof}
Select any edge $e_{ij}$ connecting vertices $v_i$ and $v_j$ and let $A = \nu(v_i) \intersect \nu(v_j)$ denote the intersection of their node label sets and $B = \nu(v_i) \union \nu(v_j)$ their union.

Define $H_{max}$ to be the subgraph of $J_n(m, m-1)$ induced by the vertex set
\[ V(H_{max}) = \left\{v \in V(J_{n}(m, m-1)): A \subset \nu(v)\right\} \]
and $H_{min}$ that induced by the vertex set
\[ V(H_{min}) = \left\{v \in V(J_{n}(m, m-1)):  \nu(v) \subset B\right\}. \]
Vertices $v_i$ and $v_j$ belong to both sets, so $e_{ij} \in H_{max}$ and $e_{ij} \in H_{min}$.

By construction, $H_{max} \in \MaxIntersecting$ and, by Proposition \ref{prop:intersection_of_clique}, there can be no other maximal clique in $\MaxIntersecting$ containing both $v_i$ and $v_j$.

Similarly, $H_{min} \in \MinIntersecting$ and, by Proposition \ref{prop:minimally_intersecting_maximal_cliques}, any maximal clique in $\MinIntersecting$ containing both $v_i$ and $v_j$ must consist of precisely all of the $m-$subsets surrounding the union of the label sets $B = \nu(v_i)\cup \nu(v_j)$.  Again there is no other such maximal clique in $\MinIntersecting$.

Theorem \ref{thm:johnson_max_clique_intersection} completes the proof by guaranteeing that there are no other possible maximal cliques containing both vertices.
%
%
\end{proof}

Although every edge, or $2$-clique, appears in one maximal clique from each of $\MinIntersecting$ and $\MaxIntersecting$, this is not the case for any other clique of size $r > 2$.  Proposition \ref{prop:johnson_clique_structure} shows that any $r$-clique, for $r > 2$, can only appear in one maximal clique, which can only be from one of  $\MinIntersecting$ or $\MaxIntersecting$.
\begin{proposition}
\label{prop:johnson_clique_structure}
Let $C_r \subset J_n(m, m-1)$ be a clique of size $r \ge 2$ with $C_r \subset H$ and $H$ a maximal clique in  $J_n(m, m-1)$.  Then, 
\begin{itemize}
\item $H \in \MinIntersecting$ only if 
$~~\bigcardinality{\bigunion_{v \in V(C_r)} \nu(v)} = m+1~~$ and 
$~~\bigcardinality{\bigintersect_{v\in V(C_r)}\nu(v)} = m+1 - r$.

\item $H \in \MaxIntersecting$ only if 
$~~\bigcardinality{\bigunion_{v \in V(C_r)} \nu(v)} = m - 1 + r~~$ and 
$~~ \bigcardinality{\bigintersect_{v\in V(C_r)}\nu(v)}  = m - 1$.
\end{itemize}
\end{proposition}
\begin{proof}
Note that every $r$-clique, $C_r$, in a graph $G$ is extendible to a maximal clique 
$H \subseteq G$.  When $G = J_n(m, m-1)$, Theorem  \ref{thm:johnson_max_clique_intersection} shows that either, $H \in \MinIntersecting$, or, $H \in \MaxIntersecting$ -- there are no other possibilities.  

Each has implications for the labels on the nodes in $C_r$. Without loss of generality, take $V(C_r) = \{v_1, \ldots, v_r\}$ as the vertices of $C_r$.

First, consider the case that $H \in \MinIntersecting$. In Remark \ref{remark:classes}, we note that every node $v_j \in V(H)$ has node label of the form $\nu(v_j) = B \setminus \{x_j\}$ for some set $B \subset \Nset{n}$, $x_j \in B$, and $\cardinality{B} = m+1$.  Further,  $B = \nu(v_i) \union \nu(v_j)$ for \textit{every} pair of distinct nodes $v_i, v_j \in V(H)$.  In particular, if $C_r$ extends to $H \in \MinIntersecting$, then 
\[
\bigcardinality{\bigunion_{v \in V(C_r)} \nu(v)} = \bigcardinality{B} = m+1 
\]
and, for some $x_1, \ldots , x_r \in B$ ($x_j$ peculiar to the node label set of each $v_j \in V(C_r)$),
\[
  \bigcardinality{\bigintersect_{v\in V(C_r)}\nu(v)} 
  =   \bigcardinality{\bigintersect_{j=1}^r (B\setminus\{x_j\}) }
  = \bigcardinality{B \setminus (\bigunion_{j=1}^{r} \{x_j\})}
  = m+1 - r,
\]
characterize the union and intersection sizes of the label sets for nodes in $V(C_r)$ when it is extendible to a maximal clique $H \in \MinIntersecting$.

Similarly, 
from Remark  \ref{remark:classes},
if $C_r$ extends to $H \in \MaxIntersecting$, then 
\[
\bigcardinality{\bigunion_{v \in V(C_r)} \nu(v)} 
= \bigcardinality{\bigunion_{j = 1}^r \left(A \union \{x_j\} \right)} 
=  \bigcardinality{A \union \left( \bigunion_{j = 1}^r  \{x_j\} \right)}
= m - 1 + r
\]
for some set $A \subset \Nset{n}$, $x_j \in \Nset{n} \setminus A$ ($x_j$ peculiar to the node label set of each $v_j \in V(C_r)$), and
\[
  \bigcardinality{\bigintersect_{v\in V(C_r)}\nu(v)} 
  =   \bigcardinality{A}
  = m - 1.
\]

\end{proof}
When $r =2$, $C_r$ is an edge and, by Proposition \ref{prop:edge_extensions}, there is both a maximal clique in $\MinIntersecting$ and one in $\MaxIntersecting$ which extend $C_r$.  This is corroborated by the matching set union sizes ($m+1$) and set intersection sizes ($m-1$) in Proposition \ref{prop:johnson_clique_structure} when $r = 2$.  
However, when $r > 2$ these sizes cannot match, and proving that for $r > 2$ any $r$-clique $C_r$  extends to a unique maximal clique in $J_n(m, m-1)$ which must be a member of one of $\MinIntersecting$ or $\MaxIntersecting$.
\begin{corollary}
\label{cor:unique_max_clique}
Let $C_r \subset J_n(m, m-1)$ be a clique of size $r >2$,  then $C_r$ can be extended to only one maximal clique $H \subset J_n(m, m-1)$.
\end{corollary}

\begin{proof}
By Proposition \ref{prop:johnson_clique_structure}, $C_r$ satisfies either $\bigcardinality{\bigunion_{v \in V(C_r)} \nu(v)} = m+1$ or $\bigcardinality{\bigintersect_{v \in V(C_r)} \nu(v)} = m-1$, but not both. Thus, there is at least one maximal clique $H$ which extends $C_r$. 

If $H$ is in $\MaxIntersecting$, then by Proposition $\ref{prop:intersection_of_clique}$, the intersection of $H$ must be the intersection of $C_r$. Similarly, if $H$ is in $\MinIntersecting$, then by Proposition $\ref{prop:minimally_intersecting_maximal_cliques}$, the union of $H$ must be the union of $C_r$.

In either case, the intersection and union of labels determine the maximal clique $H$ uniquely.
\end{proof}

Corollary \ref{cor:unique_max_clique} shows that any clique $C_r$ of size $r > 2$ can be extended to a maximal clique belonging to only one of $\MinIntersecting$ or $\MaxIntersecting$.  Proposition \ref{prop:johnson_clique_structure} provides the means for telling which one by examining the size of the intersection or of the union of the node labels of $C_r$.

\subsection{The clique partition number}
\label{sec:partition_num}
Should interest lie in the minimum number of cliques needed to partition the edges of $J_n(m, m-1)$, that is, its clique partition number $cp(J_n(m, m-1)$ \cite{erdoos1988clique}, then the maximal cliques produced by the edges in $J_n(m, m-1)$ provide the solution.  

Proposition \ref{prop:edge_extensions} showed that each edge in $J_n(m, m-1)$ led to a unique maximum clique in each of $\MinIntersecting$ and $\MaxIntersecting$.  That each edge will appear in only one element of each set and that the elements are maximal cliques, means that the cliques in either set partition the edges and that they are the fewest possible of that type.  It remains only to determine which set, $\MinIntersecting$ or $\MaxIntersecting$, is smaller -- its size will be the clique partition number.  Proposition \ref{prop:edge_extensions} thus yields the following corollary.

\begin{corollary}
\label{cor:edge_partition_num}
The clique partition number of $J_n(m, m-1)$ is given by
	\[
	cp(J_n(m, m-1)) = \min \{\cardinality{\MinIntersecting}, \cardinality{\MaxIntersecting} \}. \]
\end{corollary}
\noindent
Or, to be precise, referring to the sizes of these sets in Remark \ref{remark:classes}, with minimal rewriting, exact expressions may be had as follows.
\begin{corollary}
\label{cor:edge_partition_num_exact}
The clique partition number of $J_n(m, m-1)$ is given by
	\[
	cp(J_n(m, m-1)) = \begin{cases}
 \dbinom{n}{m-1} & n < 2m \\
 & \\
 \dbinom{n}{m+1} & n \geq 2m
\end{cases}
\]
\end{corollary}
\section{Discussion}
\label{sec:discussion}
While the results obtained above apply directly to the cliques and maximal cliques of a Johnson  graph, $J_n(m, m-1)$, they may also be expressed in terms of families of intersecting subsets of $\Nset{n}$ -- an \textit{intersecting family}, 
$\family{F}$, is a subset of the power set, 
$\powerset{n}$, whose elements are pairwise non-disjoint, that is, $A  \intersect B \neq \emptyset$ for every  $A, B \in \family{F}$ \cite<e.g., see>{erdos74extremal}.  

Numerous results have been found for intersection families \cite<e.g., see>{ExtremalSetBook}, including the celebrated Erd\H{o}s-Ko-Rado (EKR) theorem \cite{ErdosKoRado} which showed that any intersecting family having elements of size $k \leq m \leq \frac{1}{2}n$, and having no element contained in another (i.e., is an \textit{antichain}, or \textit{Sperner family}), had size at most $\binom{n -1}{m-1}$. 
The EKR  bound is obtained by the trivially intersecting family  $\family{F} \subset \mSubsetsN{n}{m}$ defined by $\family{F} = \{A \in  \mSubsetsN{n}{m} ~:~ x \in A \in \Nset{n}\}$ for some choice of $x$.  Similarly, \citeA{hiltonmilner1967} showed  that when $\family{F} \subset \mSubsetsN{n}{m}$ is further restricted to have non-null intersection, $\intersect_{A\in \family{F}} A \ne \emptyset$,  across all sets in $\family{F}$, for $2 \leq m \leq \frac{1}{2}n$, there exists another class of maximal intersecting families within $\mSubsetsN{n}{m}$ whose size is bounded above by $\binom{n-1}{m-1} - \binom{n-m-1}{m-1} +1$.

The results of the present paper are restricted to intersecting families $\family{F}_{n, m, m-1} \subset \mSubsetsN{n}{m}$ having $\cardinality{A \intersect B} = m-1$ for distinct $A, B \in \family{F}_{n, m, m-1}$.
The set of node label sets of the vertices from any maximal clique in $J_n(m, m-1)$ corresponds to a maximally intersecting family $\family{F} \subset  \family{F}_{n, m, m-1}$.

Theorem \ref{thm:johnson_max_clique_intersection} shows that there are only two possible types of such maximal intersecting families, say $\family{F}_{min}$, $\family{F}_{max} \subset  \family{F}_{n, m, m-1}$, corresponding to the two types of maximal cliques, $\MinIntersecting$, $\MaxIntersecting \subset J_n(m, m-1)$.
Expressing Remark \ref{remark:classes} in terms of intersecting families, 
there are exactly $\binom{n}{m+1}$ distinct families $\family{F} \in \family{F}_{min}$ and the following hold for each family $\family{F}$:
\begin{itemize}
\item $\cardinality{\intersect_{A\in \family{F}} A} = 0$,
\item  $\exists B \subset \Nset{n}$ of size $m+1$ with every $A \in \family{F}$ having the form $B\setminus \{x\}$ for all $x \in B$, 
\item $A_i \union A_j = B$ for all $A_i, A_j \in \family{F}$, $i \neq j$, and
\item there are $m+1$ elements in $\family{F}$,
\end{itemize}
and there are $\binom{n}{m-1}$ distinct families $\family{F} \in \family{F}_{max}$ for each of which the following hold:
\begin{itemize}
\item $\cardinality{\intersect_{A\in \family{F}} A} = m-1$,
\item  $\exists B \subset \Nset{n}$ of size $m-1$ with every $A \in \family{F}$ having the form $B \union \{x\}$ for all $x \in \Nset{n}\setminus B$, 
\item $A_i \intersect A_j = \intersect_{A \in \family{F}} A$ for all $A_i, A_j \in \family{F}$, $i \neq j$, and
\item there are $n - m+1$ elements in $\family{F}$.
\end{itemize}
Corollary \ref{cor:clique_num_johnson} gives the size of the maximal intersecting $\family{F} \subset \binom{\Nset{n}}{m}$ restricted to every pair intersection being of size $m-1$.

Similarly, Proposition \ref{prop:johnson_clique_structure} gives conditions on the size of the union and intersection of sets in a family $\family{F}_r \subset \family{F}_{n, m, m-1}$ of size $r \ge 2$ for $\family{F}_r$ to be extended to a maximally intersecting family of type $\family{F}_{min}$ or $\family{F}_{max}$ -- only one of which is possible for $r > 2$.

We again note that \citeA{shuldiner2022many} showed how intersecting families of sets are related to the partition of a set of cliques defining a clique cover for any graph.

Consider again the problem of visualizing high dimensional statistical data which served as our initial motivation.
The Johnson graph $J_n(2,1)$ has been used successfully in visual data analysis as shown by \citeA{oldford2011visual} and \citeA{hofertoldford2018}.

For a $J_n(2,1)$ graph, the maximal cliques are either a triangle representing a $3d$-space defined by the three variables in the union of the node labels ($\MinIntersecting$)\cite<e.g., see>[Figs. 2a and 3]{navGraphs2011}, or, an $(n-1)$-clique representing an $n$-dimensional space which privileges one of the $n$ variables to appear in every $2d$ node (subspace) and swaps one of the remaining variables for another whenever an edge is followed ($\MaxIntersecting$).  The latter is natural in statistics, for example, when the privileged variable might be regressed upon the second variable at each node (or vice versa).

More generally, for $J_n(m, m - 1)$, traversing maximal cliques in $\MinIntersecting$ is an exploration of an $m+1$ dimensional space via swapping one of the variables for another with every movement along an edge.  Cliques in $\MaxIntersecting$ now privilege $m-1$ variables (e.g. as regressors) while exploring the effect of changing one variable with another (e.g. as response variables in a regression model) for the remaining $n - m + 1$ variables with every movement along an edge. 

\bibliography{johnson_graph_bib.bib}

\end{document}